\documentclass[reqno,11pt]{amsart}

\IfFileExists{mymtpro2.sty}{%
  \usepackage[subscriptcorrection]{mymtpro2}
}{}

\usepackage{a4,amssymb}

\newtheorem{theorem}{Theorem}

\newtheorem{definition}[theorem]{Definition}
\newtheorem{lemma}[theorem]{Lemma}
\newtheorem{corollary}[theorem]{Corollary}

\theoremstyle{remark}

\newtheorem{myremarks}[theorem]{Remarks}

    {\end{nummer}\end{myremarks}}

\newcounter{numcount}
\newcommand{\labelnummer}{\mbox{\normalfont (\roman{numcount})}}%

\makeatletter

\newenvironment{nummer}%
  {\let\curlabelspeicher\@currentlabel%
    \begin{list}{\labelnummer}%
      {\usecounter{numcount}\leftmargin0pt%
        \topsep0.5ex\partopsep2ex\parsep0pt\itemsep0ex\@plus1\p@%
        \labelwidth2.5em\itemindent3.5em\labelsep1em%
      }%
    \let\saveitem\item%
    \def\item{\saveitem%
      \def\@currentlabel{{\upshape\curlabelspeicher}$\,$\labelnummer}}%
    \let\savelabel\label%
    \def\label##1{\savelabel{##1}%
      \@bsphack%
        \ifmmode\else%
          \protected@write\@auxout{}%
          {\string\newlabel{##1item}{{\labelnummer}{\thepage}}}%
        \fi%
      \@esphack%
    }%
  }{\end{list}}%

\renewcommand{\appendix}{\def\thesection{\textsc{Appendix}}}

 \let\leq\le
 \let\geq\ge

\let\Re\undefined 
\DeclareMathOperator{\Re}{Re}

\DeclareMathOperator{\tr}{tr\kern1pt}

\makeatletter

\newif\ifper\pertrue
\def\per{.}

\def\bti{\@ifnextchar[\bbti\bbbti}
\def\bbti[#1]#2{#2, #1.}
\def\bbbti#1{#1.}

\def\z{\@ifnextchar[\zz\zzz}
\def\zz[#1]#2#3#4#5{\perfalse\emph{#2} \textbf{#3}, #4 (#5) [#1]}
\def\zzz#1#2#3#4{\emph{#1} \textbf{#2}, #3 (#4)\ifper\per\fi\pertrue}

\def\pub{\@ifstar\pubstar\pubnostar}
\def\pubnostar{\@ifnextchar[\@@pubnostar\@pubnostar}
\def\@@pubnostar[#1]#2#3#4{#2, #3, #4, #1\ifper\per\fi\pertrue}
\def\@pubnostar#1#2#3{#1, #2, #3\ifper\per\fi\pertrue}
\def\pubstar[#1]#2#3#4{\perfalse #2, #3, #4 [#1]\pertrue}

\makeatother

\newcommand{\beq}{\begin{equation}}
\newcommand{\eeq}{\end{equation}}
\newcommand{\ba}{\begin{array}}
\newcommand{\ea}{\end{array}}
\newcommand{\bea}{\begin{eqnarray}}
\newcommand{\eea}{\end{eqnarray}}

\newcommand{\R}{\mathbb{R}}

\newcommand{\Z}{\mathbb{Z}}

\newcommand{\N}{\mathbb{N}}

\def\P{I\kern-.30em{P}}
\def\E{I\kern-.30em{E}}
\renewcommand{\E}{\mathbb{E}\mkern2mu}
\renewcommand{\P}{\mathbb{P}}

\begin{document}

\title[Monotone trace properties]{Some trace monotonicity properties and applications}

\author[J.-M.\ Combes]{Jean-Michel Combes}
\address{(J.-M.\ Combes) CPT CNRS,
Luminy Case 907,
F-13288 Marseille C\'edex 9,
France}
\email{jmcombes@cpt.uni-mrs.fr}

\author[P.\ D.\ Hislop]{Peter D.\ Hislop}
\address{(P.\ D.\ Hislop) Department of Mathematics,
    University of Kentucky,
    Lexington, Kentucky  40506-0027, USA}
\email{peter.hislop@uky.edu}

\thanks{PDH was partially supported by NSF through grant DMS-1103104 and the Universit\'e de Toulon while some of this work was done.}

\begin{abstract}
  We present some results on the monotonicity of some traces involving
functions of self-adjoint operators with respect to the natural
ordering of their associated quadratic forms.
  We also apply these results to complete a proof of the Wegner estimate for continuum models
  of random Schr\"odinger operators as given in \cite{co-hi1}.

 \end{abstract}

\maketitle \thispagestyle{empty}


\section{Statement of the Problem and Result}\label{sec:introduction}

We consider two lower-semibounded self-adjoint operators $A$ and $B$ associated with closed symmetric, lower-semibounded quadratic forms
$q_A$ and $q_B$ with form domains $Q(A)$ and $Q(B)$, respectively.
We suppose that $q_A \leq q_B$. This inequality means that $Q(B) \subset Q(A)$ and that
for all $\varphi \in Q(B)$, we have
$q_A(\varphi) \leq q_B(\varphi)$.
Under these conditions, Kato proved \cite[Theorem 2.21, chapter VI]{Kato} the following relationship between the resolvents of the two operators $A$ and $B$.  For all
$a > - \inf \sigma (A)$, we have
\beq\label{eq:qf1}
(B+a)^{-1} \leq (A+a)^{-1} .
\eeq

This resolvent inequality may be used to derive several interesting relations between the traces of functions of $A$ and $B$ under some additional assumptions. We will prove that if $f \geq 0$ and $g$ is a member of a class of functions $\mathcal{L}$ described in Definition \ref{defn:fncs1}, then
$$
{\rm Tr} (f(B) g(B)) \leq {\rm Tr} (f(B) g(A) ),
$$
see Theorem \ref{thm:funct-calc1}. We compare these inequalities with L\"owner's Theorem (see section \ref{sec:op-monotone1}) on operator monotone functions. We also use these results to complete a proof of Wegner's estimate for random Schr\"odinger operators given in \cite{co-hi1}. These results rely on the following technical theorem.

%

\begin{theorem}\label{prop:inequalities1}
Suppose that $A$ and $B$ are two lower semibounded self-adjoint operators with quadratic forms $q_A$ and $q_B$ and form domains $Q(A)$ and $Q(B)$. Suppose that $A$ and $B$ are relatively form bounded in that
\begin{enumerate}
\item the form domains satisfy $Q(B) \subset Q(A)$, and
\item for all $\psi \in Q(B)$, we have $q_A(\psi) \leq q_B(\psi)$.
 \end{enumerate}
Let $P_B$ project onto a $B$-invariant subspace so that for some $m \in \N$, and for all $a > - \inf \sigma (A)$, the operator $P_B (B+a)^{-m}$ is in the trace class.
Then we have
\begin{enumerate}
\item For all $n \in \N$ large enough so that $m < 2^n$ and for all $a > - \inf \sigma (A)$,
 $$
 {\rm Tr } (P_B(B+a)^{-2^n}) \leq {\rm Tr} ( P_B (A + a)^{-2^n});
 $$
\item For any $t > 0$,
$$
{\rm Tr } (P_B  e^{-tB}) \leq {\rm Tr}( P_B e^{- t A } ).
$$
\end{enumerate}
\end{theorem}

\begin{proof}
\noindent
1. The result of Kato \cite[Theorem 2.21, chapter VI]{Kato}, stated above,
implies that $(B+a)^{-1} \leq (A+a)^{-1}$.
We first suppose that $P_B$ is a non-zero rank one projection $P_B = P_\lambda$,
so that $B P_\lambda = \lambda P_\lambda$. From \eqref{eq:qf1}, it follows that for $a > - \inf \sigma (A)$, we have
\bea\label{eq:qf2}
{\rm Tr } P_\lambda & = & (\lambda + a) ~ {\rm Tr} ( P_\lambda (B + a)^{-1} ) \nonumber \\
 & \leq & (\lambda + a) ~ {\rm Tr} ( P_\lambda (A + a)^{-1} )\nonumber \\
 & \leq & (\lambda + a) \| P_\lambda \|_2 \|P_\lambda (A + a)^{-1} \|_2 .
 \eea
Since $P_\lambda$ is a rank one projector, we have
\beq\label{eq:rank-one1}
\| P_\lambda \|_2 = \| P_\lambda \|_1 = {\rm Tr} P_\lambda = 1,
\eeq
and
\beq\label{eq:rank-one2}
\| P_\lambda (A +a)^{-1} \|_2 = ({\rm Tr} P_\lambda (A+a)^{-2} )^{1/2}.
\eeq
Upon squaring inequality \eqref{eq:qf2} and using the results \eqref{eq:rank-one1}--\eqref{eq:rank-one2}, we obtain
\bea\label{eq:qf3}
{\rm Tr } P_\lambda & \leq  & (\lambda + a)^2 \|P_\lambda (A + a)^{-1} \|_2^2 \nonumber \\
  & = &     (\lambda + a)^2 ~ {\rm Tr} ( P_\lambda (A + a)^{-2}) .
\eea
%
%
%
We continue by rewriting the trace on the right in \eqref{eq:qf3} using the Hilbert-Schmidt norm. We square the resulting inequality, use \eqref{eq:rank-one1}--\eqref{eq:rank-one2}, and obtain
%
%
\bea\label{eq:qf5}
{\rm Tr } P_\lambda  \leq  (\lambda + a)^{2^2} ~ {\rm Tr} ( P_\lambda (A + a)^{-{2^2}}) .
\eea
Continuing in this way, we obtain for any $n \in \N$:
\beq\label{eq:qf6}
{\rm Tr } P_\lambda  \leq  (\lambda + a)^{2^n} ~ {\rm Tr} ( P_\lambda (A + a)^{-{2^n}}) .
\eeq
This may also be written as:
\beq\label{eq:qf7}
{\rm Tr } ( P_\lambda (B+a)^{-2^n} ) \leq  {\rm Tr} ( P_\lambda (A + a)^{-{2^n}}) .
\eeq
\noindent
2. We now assume that $P_B$ is a projection operator diagonalizing $B$ so that $P_B = \sum_j P_{\lambda_j}$
with $B P_{\lambda_j} = \lambda_j P_{\lambda_j}$.
If we take $2^n > m$, we can sum the inequalities \eqref{eq:qf7} over $j$ to obtain
\beq\label{eq:qf6.1}
{\rm Tr }(P_B(B+a)^{-2^n})  \leq  {\rm Tr} ( P_B (A + a)^{-{2^n}}) .
\eeq

\noindent
3. For the exponential bound, we first note that by assumption $P_B (B+a)^{-m}$ is trace class for some integer $m \geq 0$, so that $P_B e^{-tB}$ is also trace class since $(B+a)^{m} e^{-tB}$ is bounded for $t > 0$.
For $t > 0$ and $b \in \R$ so that $b  > - \inf \sigma(A)$, we obtain
\bea\label{eq:exp1}
{\rm Tr } P_Be^{-t(A+b)} &= &
 \lim_{n \rightarrow \infty} {\rm Tr }\left[ P_B \left( 1 + \frac{t (A + b)}{2^n} \right)^{-2^n} \right]   \nonumber \\
 & \geq & \lim_{n \rightarrow \infty} {\rm Tr }\left[ P_B \left( 1 + \frac{t (B + b)}{2^n} \right)^{-2^n} \right]   \nonumber \\
  & = & {\rm Tr } P_B e^{- t(B+b)},
 \eea
 where we used \eqref{eq:qf6.1} on the second line.
It follows that
 \beq\label{eq:exp2}
 {\rm Tr } (P_B(I)e^{-t(B+b)})  \leq {\rm Tr } (P_B e^{-t(A+b)}).
 \eeq
This proves the second claim.
 \end{proof}

\section{An application to trace inequalities}\label{subsec:funct-calc1}

Theorem \ref{prop:inequalities1} may be applied to a large class of functions of self-adjoint operators in order to obtain some inequalities relating traces of functions of self-adjoint operators.

\begin{definition}\label{defn:fncs1}
A real-valued function $g$ is in the class $\mathcal{L}$ if
there is a nonnegative $\sigma$-finite Borel measure $\rho$ supported on $[0, \infty)$ so that
for $s > 0$
\beq\label{eq:class2}
g(s) = \int_0^\infty e^{-st} {d \rho(t)} .
\eeq
\end{definition}


\begin{theorem}\label{thm:op-calc1}
Let self-adjoint operators $A, B$ and projector $P_B$ be as in Theorem \ref{prop:inequalities1}. Then for any $g \in \mathcal{L}$ such that $P_B g(B)$ is trace class,  one has
\beq\label{eq:op-calc1}
{\rm Tr} P_B g(B) \leq {\rm Tr} P_B g(A) .
\eeq
\end{theorem}

\begin{proof}
By the representation of $g$ in \eqref{eq:class2} and the inequality \eqref{eq:exp1} with $b=0$, one has
\bea\label{eq:laplace1}
{\rm Tr} P_B g(B) &=& \int_0^\infty {\rm Tr} ( P_B e^{-t B} ) ~d \rho (t) \nonumber \\
 & \leq & \int_0^\infty {\rm Tr} ( P_B e^{-t A} ) ~d \rho (t) \nonumber \\
 & = & {\rm Tr} ( P_B g(A) ).
 \eea
\end{proof}

A particularly useful example of functions $g$ are those related to powers of the resolvent of a self-adjoint operator.

\begin{corollary}\label{cor:loewner1}
Let self-adjoint operators $A, B$ and projector $P_B$ be as in Theorem \ref{prop:inequalities1}, and let $a > - \inf \sigma(A)$. For any $\beta > m$, where $m$ is defined in Theorem \ref{prop:inequalities1}, we have
\beq\label{eq:any-power1}
 {\rm Tr } (P_B (B + a)^{-\beta} ) \leq  {\rm Tr} ( P_B (A + a)^{-\beta}) .
\eeq
\end{corollary}

\begin{proof}
We use the Laplace transform formula valid for $\alpha > -1$ and $\Re z > 0$:
\beq\label{eq:laplace11}
\frac{1}{z^{ 1 + \alpha}} = \frac{1}{\Gamma ( 1 + \alpha)} \int_0^\infty e^{-zt} t^\alpha ~dt .
\eeq
This shows that the function $g(s) = s^{-\beta}$ is in the class $\mathcal{L}$ for any $\beta > 0$. The result
\eqref{eq:any-power1} follows from Theorem \ref{thm:op-calc1}.
\end{proof}

We conclude this section with the following generalisation of Theorem \ref{thm:op-calc1}.
It presents a trace comparison theorem for the class $\mathcal{L}$ of functions of self-adjoint operators.



\begin{theorem}\label{thm:funct-calc1}
Let $A$ and $B$ be two lower semibounded self-adjoint operators satisfying the hypotheses of Theorem \ref{prop:inequalities1}.
Suppose $g \in \mathcal{L}$
and
$f \geq 0$ so that $f(B)g(B)$ is trace class.
We then have
\beq\label{eq:comparison1}
{\rm Tr } (f(B)g(B)) \leq {\rm Tr}( f(B) g(A)) .
\eeq
\end{theorem}

\begin{proof}
Since $f(B)g(B)$ is assumed to be trace class, the operator $B$ must have pure point spectrum $\{ \lambda_j \}$ on the support of the function $fg$. For any $j$, it follows from Theorem \ref{thm:op-calc1} that
\beq\label{eq:comparison3}
{\rm Tr } (P_{\lambda_j} g(B)) \leq {\rm Tr}( P_{\lambda_j} g(A)),
\eeq
where, as above, $B P_{\lambda_j} = \lambda_j P_{\lambda_j}$. Multiplying both sides of \eqref{eq:comparison3}
by $f(\lambda_j) \geq 0$, and summing over $j$ results in \eqref{eq:comparison1}.
\end{proof}

We remark that if $g(B)$ is in the trace class, we may take $f = 1$ and obtain
\beq\label{eq:comparison2}
{\rm Tr } (g(B)) \leq {\rm Tr}( g(A)),
\eeq
a result that also follows from the Min-Max Theorem since any function $g \in \mathcal{L}$ is decreasing.


\section{A relation with operator monotone functions}\label{sec:op-monotone1}

The following class of functions was introduced by L\"owner \cite{lowner} in 1934 and is the object of his famous theorem that we now recall.

\begin{definition}\label{defn:op-monotone1}
Let $J \subset \R$ be a finite interval or a half-line. A function $g:J \rightarrow \R$ is operator monotone increasing (respectively, decreasing) in $J$ if for all pairs of self-adjoint operators $A,B$ with spectrum in $J$ the operator inequality $A \leq B$ implies the operator inequality $g(A) \leq g(B)$ (respectively $g(B) \leq g(A)$.)
\end{definition}

If $g$ is operator monotone decreasing, then \eqref{eq:comparison1} holds for any $f \geq 0$ and for all pairs of operators $A \leq B$ such that $f(B)g(B)$ is trace class. Because of this, we study the relationship between operator monotone decreasing functions and the class $\mathcal{L}$ of Definition \ref{defn:fncs1}.
For simplicity, we assume that $J = \R^+$.
We denote by $\mathcal{I}$ (respectively, $\mathcal{D}$) the class of operator monotone increasing (respectively, decreasing) functions on $\R^+$. The map $g \in \mathcal{D} \rightarrow \tilde{g} \in \mathcal{I}$ defined by
$\tilde{g}(s) := g(1/s)$, for $s > 0$, is a bijective involution between $\mathcal{D}$ and $\mathcal{I}$.

L\"owner's Theorem \cite{lowner} (see also \cite{Bendat-Sherman,Bhatia-Sinha,chansangiam,donoghue,hansen}) states that $g$ is operator monotone increasing if and only if $g$ has an analytic extension to the upper-half complex plane with a positive imaginary part. Such functions are known as Herglotz or Pick functions and have integral representations. For example, Hansen \cite{hansen} proved the following representation.

\begin{lemma}\cite[Corollary 5.1]{hansen}\label{lemma:representation1}
Let $\tilde{g}$ be a positive operator monotone increasing function on the half-line $\R^+$. Then there exists a bounded, positive measure $\mu$ on $\R^+$ such that
\beq\label{eq:repr1}
\tilde{g}(s) = \int_{\R^+} \frac{s (1+ \lambda)}{s + \lambda} ~d \mu (\lambda) , ~~ s > 0.
\eeq
\end{lemma}

It follows easily from
Kato's result \eqref{eq:qf1} that any function on $\R^+$ with a representation as on the right of
\eqref{eq:repr1} is in the class $\mathcal{I}$. The difficult part of the proof of L\"owner's Theorem is the converse.

Using the bijection $g \rightarrow \tilde{g}$ between $\mathcal{D}$ and $\mathcal{I}$ described above, it follows that if $f \in \mathcal{D}$, then $f$ has a representation of the form
\beq\label{eq:repr2}
f(s) = \int_{\R^+} \frac{1+ \lambda}{ 1 + s  \lambda} ~d \mu (\lambda) , ~~ s > 0 ,
\eeq
for some bounded positive measure $\mu$ on $\R^+$. Using the Laplace transform representation \eqref{eq:laplace1} with $\alpha = 0$, we may write $f$ as
\beq\label{eq:laplace2}
f(s) = \int_{\R^+} e^{-st} h(t) ~dt,
\eeq
where $h$ is defined by
\beq\label{eq:laplace3}
h(t) = \int_{\R^+} \left( 1 + \frac{1}{\lambda} \right) e^{- \frac{t}{\lambda}} ~d \mu (\lambda) .
\eeq
The function $h \in L^1_{\rm loc} (\R^+)$ so by Definition \ref{defn:fncs1}, the function $f \in \mathcal{L}$.

This shows that $\mathcal{D} \subset \mathcal{L}$. On the other hand, the functions on $\R^+$
such as $f(s) = e^{- a s}$, with $a > 0$, or $f(s) = (s + a )^{- \rho}$, with $\rho > 1$ and $a > 0$, belong to the class $\mathcal{L}$ but not to the class $\mathcal{D}$. As a consequence, we obtain the following theorem.

\begin{theorem}\label{thm:different-classes1}
The class of operator monotone decreasing functions $\mathcal{D}$ is strictly contained in the class of functions $\mathcal{L}$.
\end{theorem}

\section{An application to the proof of Wegner's estimate}

We complete the proof of the Wegner estimate given in \cite{co-hi1}. Since this method of proof seems to have been used in several subsequent papers, we wanted to present the complete argument.
 The Wegner estimate is an upper bound on the probability that a local Hamiltonian has eigenvalues in a given energy interval.
We considered a large cube $\Lambda$ centered at the origin in $\R^d$ with odd integer side length.
We let $H_\omega := - \Delta + V_\omega$ be the random Schr\"odinger operator on $L^2 (\R^d)$ with a standard Anderson-type random potential $V_\omega \geq 0$ (this condition can be removed). We denote by $H_\Lambda$
the restriction of $H_\omega$ to $\Lambda$ with Dirichlet boundary conditions.
This operator has discrete spectrum. For any bounded interval $I = [ I_-, I_+] \subset \R$, we let $E_\Lambda (I)$ be the spectral projection for $H_\Lambda$ and interval $I$. The trace of this projection is finite and it is a random variable. The Wegner estimate proved in \cite[Proposition 4.5]{co-hi1} is
\beq\label{eq:wegner1}
\P \{ {\rm Tr} E_\Lambda (I) \geq 1 \} \leq C_W |I| |\Lambda|,
\eeq
where $C_W >0$ is a finite constant depending on $I_+$.

The proof of the Wegner estimate in \cite{co-hi1} depends on a comparison of the
operator $H_\Lambda$ to a direct sum of operators defined on unit cubes in $\Lambda$.
Let $\Lambda = {\rm Int} \overline{ \{ \cup_{j} \Lambda_1 (j) \}}$
be a decomposition of $\Lambda$ into unit cubes centered at the lattice points $\tilde{\Lambda}$ of $\Lambda$.
In the proof of Proposition 4.5 \cite[section 4]{co-hi1}, we used the  operator inequality
\beq\label{eq:ineq1}
H_\Lambda \geq H_{N, \Lambda} \equiv - \oplus_j \Delta_{N,j} ,
\eeq
where $- \Delta_{N,j}$ is the Neumann Laplacian on a unit cube centered at $j \in \Lambda \cap \Z^d$, if the boundary of the cube does not intersect the boundary of $\Lambda$, or the Laplacian with mixed Neumann-Dirichlet boundary conditions if
the cube's boundary intersects the boundary of $\Lambda$. This inequality is valid only in the operator form sense. It cannot be used in conjunction with Jensen's inequality as done after equation (4.15) in \cite{co-hi1} since the eigenfunctions $\phi_n$ of $H_\Lambda$ are not in the operator domain of $H_{N, \Lambda}$.


We apply Theorem \ref{prop:inequalities1} in order to complete the proof of Wegner's estimate as stated
in \cite[Proposition 4.5]{co-hi1}. We divide the set of indices $\tilde{\Lambda}$ of the unit cubes in $\Lambda$ into two sets: The set $\partial \tilde{\Lambda}$
associated with unit cubes whose boundary intersects $\partial \Lambda$, and the set
${\rm Int} \tilde{\Lambda} $ of interior points.
We take $A = H_{N, \Lambda}$, as defined in \eqref{eq:ineq1},
and $B = H_{ \Lambda}$, the restriction of $H$ to $\Lambda$ with Dirichlet boundary conditions.

We verify conditions (1) and (2) of Theorem \ref{prop:inequalities1}.
As quadratic forms, we have $Q(B) := Q(H_\Lambda) = H_0^1(\Lambda)$, whereas
  $Q(A) := Q(H_{N,\Lambda}) = \{ \oplus_{j \in {\rm Int} \tilde{\Lambda}} H^1(\Lambda_1(j)) \} \oplus
  \{ \oplus_{j \in \partial \tilde{\Lambda}} H_M^1(\Lambda_1(j)) \}$, where $H_M^1 (\Lambda_1(j))$ consists of functions in $H^1(\Lambda_1(j))$ with Neumann boundary conditions
  along $\partial \Lambda \cap \partial \Lambda_1(j)$.
    It follows that $Q(H_{\Lambda}) \subset Q(H_{N, \Lambda})$. The second condition of Theorem \ref{prop:inequalities1} holds identically.

We have $\inf \sigma (A) = 0$ in this case. Then, with the notation of \cite{co-hi1}, the projection $P_B$ is $E_\Lambda (I_\eta)$.
From part 2 of Theorem \ref{prop:inequalities1}, we have
\bea\label{eq:wegner2}
{\rm Tr} E_\Lambda (I_\eta) & \leq & e^{ I_{\eta, +}  } {\rm Tr} ( E_\Lambda (I_\eta) e^{- H_{\Lambda} } )  \nonumber \\
  & \leq & e^{ I_{\eta, +}  } {\rm Tr} ( E_\Lambda (I_\eta) e^{- H_{N, \Lambda} } )  \nonumber \\
   & = & e^{ I_{\eta, +}  } \left( \sum_{j \in \Lambda \cap \Z^d} {\rm Tr} ( E_\Lambda (I_\eta) e^{\Delta_{N,j} } \chi_j ) \right) ,
\eea
where $\chi_j$ is the characteristic function for the unit cube $\Lambda_1(j)$ centered at $j \in \Z^d$.
In this way, we recover (4.16) of \cite{co-hi1}. Following the remainder of the proof there, since the operators $- \Delta_{N,j}$ do not depend on the random variables, we expand the trace in the eigenfunctions of $- \Delta_{N, j}$ and apply the spectral averaging result \cite[Corollary 4.2]{co-hi1}. In this manner, one obtains \eqref{eq:wegner1}. We refer the reader to \cite{CHK:2007} for a more general proof of the Wegner estimate.

\end{document}